\documentclass[a4paper,11pt]{amsart}
\usepackage{amsmath,amsthm,amssymb,amsfonts,enumerate,color,hyperref,esint}

\newcommand{\R}{\mathbb{R}}

\newcommand{\C}{\mathbb{C}}

\newcommand{\N}{\mathbb{N}}

\renewcommand{\Re}{\operatorname{Re}}

\newcommand{\vep}{\varepsilon}

\DeclareMathOperator{\dive}{div}

\newtheorem{thm}{Theorem}[section]

\newtheorem{lem}[thm]{Lemma}

\theoremstyle{definition}

\numberwithin{equation}{section}

\allowdisplaybreaks

\author[A. Biswas]{Animesh Biswas}
\address{Department of Mathematics \\
University of Nebraska-Lincoln \\
210 Avery Hall, Lincoln \\
NE 68588, United States of America}
\email{abiswas2@unl.edu}

\author[P. R. Stinga]{Pablo Ra\'ul Stinga}
\address{Department of Mathematics \\
Iowa State University \\
396 Carver Hall, Ames \\
IA 50011, United States of America}
\email{stinga@iastate.edu}



\thanks{Research partially supported by Simons Foundation grant 580911.}
  
\begin{document}

\title[Sharp extension problems in Banach spaces]{Sharp extension problem characterizations for higher fractional power operators in Banach spaces}

\begin{abstract}
We prove sharp characterizations of higher order fractional powers $(-L)^s$, where $s>0$ is noninteger, of
generators $L$ of uniformly bounded $C_0$-semigroups on Banach spaces via extension problems,
which in particular include results of Caffarelli--Silvestre, Stinga--Torrea and Gal\'e--Miana--Stinga when $0<s<1$.
More precisely, we prove existence and uniqueness of solutions $U(y)$, $y\geq0$, to initial value problems
for both higher order and second order extension problems
and characterizations of $(-L)^su$, $s>0$, in terms of boundary derivatives of $U$ at $y=0$, under the sharp
hypothesis that $u$ is in the domain of $(-L)^s$. Our results resolve the question of setting up the correct initial conditions
that guarantee well-posedness of both extension problems. Furthermore, we discover new explicit
subordination formulas for the solution $U$ in terms of the semigroup $\{e^{tL}\}_{t\geq0}$ generated by $L$.
\end{abstract}

\maketitle

\section{Introduction}

Extension problem characterizations of fractional powers of linear operators
\cite{Caffarelli-Silvestre,Gale-Miana-Stinga, Stinga-Torrea} are powerful tools in the study of nonlocal fractional equation problems
in analysis, PDEs, geometry, fractional calculus, mathematical finance, continuum mechanics,
numerical analysis and computational mathematics, among other areas.
The extension problem characterization for the fractional Laplacian $(-L)^s=(-\Delta)^s$ in $\R^n$,
for $0<s<1$, was introduced by Caffarelli and Silvestre \cite{Caffarelli-Silvestre} in 2007. In 2010, 
Stinga and Torrea developed in \cite{Stinga-Torrea} the method of semigroups and generalized the extension problem
characterization to fractional powers $(-L)^s$, $0<s<1$, of any nonnegative normal operator $-L$ in Hilbert spaces.
The most general extension description was proved in 2013 by Gal\'e, Miana and Stinga, see \cite{Gale-Miana-Stinga}.
The result in \cite{Gale-Miana-Stinga} is established through the method of semigroups for infinitesimal generators $L$ of
tempered $\alpha$-times integrated semigroups on Banach spaces $X$, for $\alpha\geq0$. In particular, it applies to
generators $L$ of uniformly bounded $C_0$-semigroups $\{e^{tL}\}_{t\geq0}$ on $X$ and,
in this case, it states the following. Given $u\in X$ and $0<s<1$, define
\begin{equation}\label{eq:Usemigroup}
U(y)=U[u](y)=\frac{y^{2s}}{4^s\Gamma(s)}\int_0^\infty e^{-y^2/(4t)}e^{tL}u\,\frac{dt}{t^{1+s}}\qquad y>0.
\end{equation}
Then $U$ is a bounded classical solution to the $X$-valued extension problem
\begin{equation}\label{eq:extension_problem_X}
\begin{cases}
LU+\frac{1-2s}{y}\partial_yU+\partial_{yy}U=0&\hbox{for}~y>0\\
\lim_{y\to0}U(y)=u&\hbox{in}~X.
\end{cases}
\end{equation}
If, in addition, $u\in D(L)$ (the domain of $L$ in $X$) then
\begin{equation}\label{eq:normalderivative}
-\lim_{y \to 0}y^{1-2s}\partial_yU(y)=c_s(-L)^su\qquad\hbox{in}~X
\end{equation}
where $c_s>0$ is a constant explicitly computed in \cite{Stinga-Torrea} that depends only on $s$.
In fact, $U$ given by \eqref{eq:Usemigroup} is the unique bounded classical solution to \eqref{eq:extension_problem_X}
and if $u\in D((-L)^s)$, for $0<s<1$, then \eqref{eq:normalderivative} still holds, see \cite{Meichsner-Seifert}.
Formula \eqref{eq:Usemigroup} was first discovered in \cite{Stinga-Torrea}.

In this paper we prove sharp characterizations for all fractional powers $(-L)^s$,
where $s>0$ is noninteger, with both higher order and second order extension problems, see 
\eqref{eq:uniqueness s greater 2} and \eqref{eq:bvpinyalls}, respectively. 
We show existence and uniqueness of the classical solution to the these
extension problems and the characterizations of $(-L)^su$ as certain derivatives of $U$ at $y=0$ under the sharp
hypothesis that $u\in D((-L)^s)$, see \eqref{eq:introNeumannL} and \eqref{eq:introNeumanny}. In particular, we set up the correct initial
conditions for well-posedness of the problems. Furthermore, we find new explicit representations of the solution $U$
in the form of subordination formulas involving the semigroup $\{e^{tL}\}_{t\geq0}$, see \eqref{eq:U component s Greater 2}.
Our main result is the following (see Section \ref{sec:balakrishnan} for notation).

\begin{thm}[Extension problems for any noninteger $s>0$]\label{thm:extension-general}
Let $L$ be the infinitesimal generator of a uniformly bounded $C_0$-semigroup $\{e^{tL}\}_{t\geq0}$ on a Banach space $X$.
Assume that $0\in\rho(L)$, the resolvent set of $L$.
Let $s>0$ be any noninteger and let $[s]$ be the integer part of $s$.
Given any $u\in X$, let $U(y)=U[u](y)$ be as in \eqref{eq:Usemigroup}.
The following statements hold.
\begin{enumerate}[$(a)$]
\item For any $k\geq0$, $U\in C^\infty((0,\infty);D(L^k))\cap C([0,\infty);X)$
and $U$ is a bounded classical solution to the higher order extension problem
\begin{equation}\label{eq:extension_problem_s_greater1}
\begin{cases}
\big(L+\frac{1-2(s-[s])}{y}\partial_y+\partial_{yy}\big)^{[s]+1}U=0&\hbox{for}~y>0\\
\lim_{y\to0}U(y)=u&\hbox{in}~X
\end{cases}
\end{equation}
and to the second order extension problem
\begin{equation}\label{eq:introextension}
\begin{cases}
LU+\frac{1-2s}{y}\partial_yU+\partial_{yy}U=0&\hbox{for}~y>0\\
\lim_{y\to0}U(y)=u&\hbox{in}~X.
\end{cases}
\end{equation}
\item If $f\in X$ and $u\in D((-L)^s)$ is the unique solution to $(-L)^su=f$ then we have
\begin{equation}\label{eq:U component s Greater 2}
\begin{aligned}
U(y) &= \sum^{[s]}_{k=0} \frac{y^{2k}\Gamma(s-k)}{4^kk!\Gamma(s)}L^ku
+\frac{1}{\Gamma(s)} \int^\infty_0\bigg[e^{-y^2/(4t)}- \sum^{[s]}_{k=0} \frac{(-1)^ky^{2k}}{k!(4t)^k}\bigg] e^{tL}f\,\frac{dt}{t^{1-s}} \\
&= \sum^{[s]}_{k=0} \frac{y^{2k}\Gamma(s-k)}{4^kk!\Gamma(s)}L^ku
+\frac{y^{2s}}{4^s\Gamma(s)}\int^\infty_0\bigg[e^{-r}-\sum^{[s]}_{k=0} \frac{(-r)^k}{k!}\bigg]e^{\frac{y^2}{4r}L}f\,\frac{dr}{r^{1+s}}.
\end{aligned}
\end{equation}
\item Furthermore, $u\in D((-L)^s)$ if and only if the limits
$$\lim_{y\to0}y^{1-2(s-[s])}\partial_y\big(L+\tfrac{1-2(s-[s])}{y}\partial_y+\partial_{yy}\big)^{[s]}U(y)$$
or
\begin{equation}\label{eq:U Ls Greater 2}
\lim_{y\to0}y^{1-2(s-[s])}\partial_y\big(\tfrac{2}{y}\partial_y\big)^{[s]} U(y)
\end{equation}
exist in $X$ and, in these cases,
\begin{equation}\label{eq:introNeumannL}
\lim_{y\to0}y^{1-2(s-[s])}\partial_y\big(L+\tfrac{1-2(s-[s])}{y}\partial_y+\partial_{yy}\big)^{[s]}U(y)
=c_s[s]!(-L)^su
\end{equation}
and
\begin{equation}\label{eq:introNeumanny}
\lim_{y\to0}y^{1-2(s-[s])}\partial_y\big(\tfrac{2}{y}\partial_y\big)^{[s]}U(y)
=c_s(-L)^su
\end{equation}
where $c_s=\frac{(-1)^{[s]+1}\Gamma([s]+1-s)}{4^{s-([s]+1/2)}\Gamma(s)}$.
\item If $u\in D((-L)^s)$ then $U$ is the unique classical solution to the higher order initial value
extension problem
\begin{equation}\label{eq:uniqueness s greater 2}
\begin{cases}
\big(L+\frac{1-2(s-[s])}{y}\partial_y+\partial_{yy}\big)^{[s]+1}U=0&\hbox{for}~y>0\\
\lim_{y\to0}\big(L+\frac{1-2(s-[s])}{y}\partial_y+\partial_{yy}\big)^mU =\frac{[s]!\Gamma(s-m)}{([s]-m)!\Gamma(s)}
L^mu&\hbox{for}~0\leq m\leq [s] \\
\lim_{y\to0}y^{1-2(s-[s])}\partial_y \big(L+\frac{1-2(s-[s])}{y}\partial_y+\partial_{yy}\big)^mU=0&\hbox{for}~0\leq m<[s] \\
\lim_{y\to0}y^{1-2(s-[s])}\partial_y \big(L+\frac{1-2(s-[s])}{y}\partial_y+\partial_{yy}\big)^{[s]}U=c_s[s]!(-L)^su,
\end{cases}
\end{equation}
and to the second order initial value extension problem
\begin{equation}\label{eq:bvpinyalls}
\begin{cases}
LU+\frac{1-2s}{y}\partial_yU+\partial_{yy}U=0&\hbox{for}~y>0\\
\lim_{y\to0}\big(\frac{2}{y}\partial_y\big)^mU(y) =\frac{\Gamma(s-m)}{\Gamma(s)}L^mu&\hbox{for}~0\leq m\leq [s] \\
\lim_{y\to0}y^{1-2(s-[s])}\partial_y \big(\frac{2}{y}\partial_y\big)^mU(y)=0&\hbox{for}~0\leq m<[s] \\
\lim_{y\to0}y^{1-2(s-[s])}\partial_y \big(\frac{2}{y}\partial_y\big)^{[s]}U(y)=c_s(-L)^su.
\end{cases}
\end{equation}
\end{enumerate}
\end{thm}

By classical solution here we mean a solution $U(y)$ that is as continuously differentiable in $(0,\infty)$ as the equation
requires with values in the corresponding domains of $L$ and integer powers of $L$, and whose initial conditions
are satisfied with continuity up to $y=0$.

One of the main difficulties of the assumption $u\in D((-L)^s)$ in Theorem \ref{thm:extension-general} is the actual semigroup description
of the domain of $(-L)^s$. Indeed,
Berens, Butzer and Westphal proved in \cite{Berens-Butzer-Westphal} that $u\in D((-L)^s)$ if and only if,
for any integer $k>s$, the limit
$$v:=\lim_{\vep\to0}\frac{1}{c(s,k)}\int_\vep^\infty\big(e^{tL}u-u\big)^k\,\frac{dt}{t^{1+s}}$$
exists in $X$, where
$$c(s,k)=\int_0^\infty\big(e^{-t}-1\big)^k\,\frac{dt}{t^{1+s}}$$
and, in this case, $v=(-L)^su$.
In general, this limit has no explicit rate of convergence, the integral does not converge absolutely
(however, it does if $u$ is more regular, say, $u\in D((-L)^{[s]+1})$, see \cite{Gale-Miana-Stinga} where this assumption is used when $0<s<1$)
and may even oscillate, involving cancelations.
Hence we can not directly use this description in our semigroup analysis. We overcome these difficulties
thanks to our new semigroup formulas \eqref{eq:U component s Greater 2}.

Another problem that we solve and clarify is that of imposing appropriate initial conditions to \eqref{eq:extension_problem_s_greater1}
and \eqref{eq:introextension} that guarantee well-posedness of the extension problem. Indeed, \eqref{eq:extension_problem_s_greater1}
involves a $(2[s]+2)$-order $X$-valued ODE.  Thus, it is natural to impose $(2[s]+2)$ initial conditions that will ensure uniqueness of solutions.
This is indeed the case as we prove here, see \eqref{eq:uniqueness s greater 2}
and Theorem \ref{thm:uniquenessgeneral}. The case $0<s<1$ is revisited in our context in Theorem \ref{lem:unique_A_unbounded}.

On the other hand, \eqref{eq:introextension} is a second order $X$-valued ODE problem,
so one may think that imposing an extra initial Neumann condition involving $U'(y)$ at $y=0$ would suffice for uniqueness. However, we show that
the only case in which uniqueness can be achieved in general in \eqref{eq:introextension}
with a Neumann-type initial condition is when $0<s<1$, see Lemma \ref{lem:Besselivp}. When $s>1$, we prove
uniqueness for the second order initial value problem \eqref{eq:introextension} with $(2[s]+2)>2$ initial conditions \eqref{eq:bvpinyalls}, see 
Theorem \ref{thm:uniquenessgeneral}.

Observe that considering the initial value problems \eqref{eq:uniqueness s greater 2} and \eqref{eq:bvpinyalls}
instead of trying to impose conditions for $y$ near infinity in \eqref{eq:extension_problem_s_greater1}
and \eqref{eq:introextension} to prove well-posedness is enough for analytical, numerical and computational applications.
Indeed, the extension problem is most significant when $y\to 0$.

Our analysis is finally able to clarify the role of initial conditions in extension problems
for higher fractional power operators. In fact, extension problems of this kind have been considered in recent years.
For Hilbert spaces, the first result was proved by Roncal and Stinga in \cite{Roncal-Stinga}, where it was shown that $U[u](y)$
solves \eqref{eq:introextension} and satisfies \eqref{eq:introNeumanny}, for any $s>0$ noninteger.
For the fractional Laplacian in $\R^n$, Yang in \cite{Yang} applied the Fourier transform
to prove a characterization of $(-\Delta)^s$, $s>0$ noninteger, through a higher order extension equation
as in \eqref{eq:extension_problem_s_greater1} in adequate Sobolev spaces and satisfying
a number of initial conditions that are a mix between some of those in
\eqref{eq:uniqueness s greater 2} and some from \eqref{eq:bvpinyalls}.
Later on, Cora and Musina in \cite{Cora-Musina}
further expanded the results of Yang by using variational methods and the Poisson kernel from \cite{Caffarelli-Silvestre}.
In particular, \cite{Cora-Musina} shows that minimizers of the corresponding energy, which
imposes conditions on $U$ as $y\to\infty$, are unique. They also prove various properties of derivatives of $U$ at $y=0$.
However, it is not clear from their work which are the initial conditions that would suffice for uniqueness.

More recently, Musina and Nazarov proved in \cite{Musina-Nazarov} an extension problem characterization
of higher fractional powers $(-L)^s$ for nonnegative symmetric operators $-L$ on Hilbert spaces
with an extension equation like in \cite{Cora-Musina,Yang}. For this, they used spectral separation of
variables and Bessel functions, extending the methodology initially put forth in \cite{Stinga-Torrea} for $0<s<1$.
The results of our paper show that, under the adequate initial conditions that we establish,
the extension characterizations for Hilbert spaces of \cite{Roncal-Stinga}
and \cite{Musina-Nazarov} are equivalent. In turn, we provide new semigroup subordination formulas for the solution $U$
that are not present in \cite{Cora-Musina,Musina-Nazarov,Yang}, and generalize to Banach spaces
by completely different techniques.

Although fractional powers of linear operators on Banach spaces is a fairly classical topic in functional analysis
and operator theory \cite{Balakrishnan,Butzer-Berens-Book,Martinez,Martinez-book,Yosida}, for the past 15 to 20 years
the mathematics, physics, engineering, biology and computer science
communities have shown a still increasing interest in the theory and applications of 
nonlinear nonlocal problems involving fractional power operators. Higher fractional powers $s>1$
are central in many pure and applied problems. We only mention one application to fluid mechanics here.
Let $v=(v_1(x_1,x_2),v_2(x_1,x_2))$ solve
the $2/3$-fractional Stokes equation of anomalous turbulence
$(-\Delta)^{2/3}v+\nabla p=f$, with $\dive v=0$ in $\R^2$,
for a given vector field $f=(f_1,f_2)$ and a scalar function $p$, see \cite{Chen}.
Since $v$ is divergence free, one can introduce a stream function $u=u(x_1,x_2)$ set by $v_1=-\partial_{x_2}u$ and $v_2=\partial_{x_1}u$.
By letting $g:=\partial_{x_2}f_1-\partial_{x_1}f_2$
it easily follows that $u$ is a solution to the higher order fractional equation $(-\Delta)^{1+2/3}u=g$.
In addition, we especially point out that extension problems have been crucial
for the numerical analysis and computational implementation of fractional nonlocal equations. Indeed,
the seminal work \cite{Nochetto-et-al} generates finite element approximations by using the extension and formulas of \cite{Stinga-Torrea}.

Our results are general and can be applied, for instance, to higher fractional powers of second order elliptic operators, parabolic operators,
hyperbolic operators, Laplacians on manifolds and graphs, among many others. They will be useful for numerical implementation.
Furthermore, our theorems can be extended to generators $L$ of tempered $\alpha$-times integrated semigroups, for $\alpha\geq0$,
to fractional complex powers $s\in\C\setminus\N$ with positive real part $\Re(s)>0$, and the solution $U(y)$
can be analytically extended to a complex sector including the half line $(0,\infty)$.
These will appear elsewhere.
We also mention here that the condition $0\in\rho(L)$
is not restrictive for applications. For example, for the fractional Neumann
Laplacian $(-\Delta_N)^s$ in a bounded domain or the fractional Ornstein--Uhlenbeck operator $(-\Delta+2x\cdot\nabla)^s$
in the whole space $\R^n$, where the first eigenvalue is $0$,
one looks for solutions in quotient spaces over constants so that uniqueness holds.

The paper is organized as follows. Section \ref{sec:balakrishnan} contains preliminary facts about fractional powers of nonnegative closed operators
and of generators of uniformly bounded $C_0$-semigroups
on Banach spaces. In particular, Theorems \ref{lem:approx_identity} and \ref{lem:inverse} will be essential to prove our main results.
In Section \ref{sec:extension} we analyze the explicit semigroup subordination formula of the solution to the extension equation \eqref{eq:Usemigroup}.
Uniqueness is addressed in Section \ref{sec:uniqueness}.
In Section \ref{sec: extension s small} we present and prove
Theorem \ref{thm:extension-general} for the cases $0<s<1$ and $1<s<2$.
We do this for two reasons. First, they are the cases that are mostly used in applications
and for which we can say more about characterizing $(-L)^su$ in terms of limits
of incremental quotients of $U$, see \eqref{eq:neumanns} and \eqref{eq:1s2incremental}. The latter will be helpful
when performing finite difference approximations of $U$. Second, we believe that it will help the
reader in understanding the structure of the proof of the general Theorem \ref{thm:extension-general},
which is done in Section \ref{sec: extension s general}.

\section{Fractional power operators in Banach spaces}\label{sec:balakrishnan}

Throughout the paper, $X$ denotes a Banach space with norm $\|\cdot\|_X$, $I$ is the identity operator
and $A:D(A)\subset X\to X$ is a linear operator with domain $D(A)$ and range $R(A)$.
The resolvent set $\rho(A)$ of $A$ is the set of all $\lambda\in\C$ such that
$R(\lambda I - A)$ is dense in $X$ and $(\lambda I - A)^{-1}$ is a bounded operator on its domain $D((\lambda I-A)^{-1})=R(\lambda I -A)$. 
Hence, $(\lambda I -A)^{-1}$ extends as a bounded linear
operator on $X$. The spectrum of $A$ is $\sigma(A)=\C\setminus\rho(A)$. We say that $A$ is nonnegative if
$(-\infty,0)\subset\rho(A)$ and
$$M_A:=\displaystyle\sup_{\mu>0}\|\mu(\mu I +A)^{-1}\|<\infty.$$
In this case, we call $M_A$ the nonnegativity constant of $A$.

From now on, and for the rest of the paper, we assume that $A$ is a nonnegative operator.

\subsection{Fractional powers of nonnegative operators}

The construction of fractional powers of $A$ is classical, see \cite{Balakrishnan,Martinez,Martinez-book,Yosida}.

For $s>0$, consider first the Balakrishnan operators $J^s$ defined as follows. For $0<s<1$, $D(J^s)=D(A)$ and
$$J^s u = \frac{\sin(s\pi)}{\pi} \int^\infty_0 \mu^{s-1} (\mu I +A)^{-1} Au\, d\mu.$$
For $0<s<2$, $D(J^s) = D(A^2)$ and
$$J^s u = \frac{\sin(s\pi)}{\pi} \int^\infty_0 \mu^{s-1} \bigg[(\mu I +A)^{-1} - \frac{\mu}{1+\mu^2} \bigg]Au \, d\mu + \sin(s \pi/2)Au.$$
For $n<s<n+1$, $D(J^s)=D(A^{n+1})$ and $J^s u = J^{s-n}A^n u$.
Finally, for $n<s\leq n+1$, $D(J^s) = D(A^{n+2})$ and $J^s u = J^{s-n}A^n u$.

We next define the positive fractional powers of $A$.
If $A$ is bounded and $s>0$ then $A^s=J^s$, with domain $D(A^s)=D(J^s)=D(A)=X$.
If $A$ is unbounded and $0 \in \rho(A)$ then $A^s = [(A^{-1})^s]^{-1}$, with domain $D(A^s)=R((A^{-1})^s)$.
If $A$ is unbounded and $0 \in \sigma(A)$ then 
$$A^s u = \lim_{\varepsilon \to 0} (\varepsilon I+A)^s u.$$
The domain of $A^s$ in this last case is the collection of all $u$'s for which the limit exists.

If $A$ is bounded then $A^s$ is bounded. If 
$A$ is unbounded, $A^s$ is closed and $D(A^s)\subset\overline{D(A)}$.
Also, $A^1=A$.

If $0\in\rho(A)$ then $(A^{-1})^s$ is injective.
We can then consider negative fractional powers of $A$ as well.
Indeed, in this case, $A^{-s} = (A^{-1})^s$ with domain $D(A^{-s})=D((A^{-1})^s)=X$.
We have that $A^{-s}=(A^s)^{-1}$.

One of the main results we will need is the following consequence of Propositions 7.4.1 and 7.2.2 of \cite{Martinez-book},
in which it is important to note the presence of the hypothesis $u\in D(A^\alpha)$.

\begin{thm}\label{lem:approx_identity}
Let $A$ be a nonnegative operator with $0\in\rho(A)$. Then, for all $u\in D(A^\alpha)$ and $0\leq\beta\leq\alpha$,
$$ \lim_{\varepsilon\to0} (\varepsilon I+A)^{-\beta} A^\alpha u = A^{\alpha-\beta}u.$$ 
\end{thm}

\subsection{Fractional powers of generators}

Here we collect classical facts about semigroups of linear operators
and their infinitesimal generators, see \cite{Butzer-Berens-Book,Pazy,Yosida}.

A family $\{S_t\}_{t \geq 0}$ of bounded linear operators on $X$ is a semigroup if 
$S_0 = I$ and $S_{t_1}\circ S_{t_2} = S_{t_1+t_2}$ for every $t_1,t_2\geq0$.
If, in addition, $S_tu \to u$ in $X$ as $t\to 0$ for all $u \in X$, then we say that
$\{S_t\}_{t \geq 0}$ is a $C_0$-semigroup.
If $\{S_t\}_{t\geq0}$ is a $C_0$-semigroup then there exist constants $\omega\geq0$ and $M\geq1$
such that the operator norm of $S_t$ satisfies the estimate $\|S_t\|\leq Me^{\omega t}$,
for all $t\geq0$. If $\omega=0$, that is,
$\|S_tu\|_X\leq M\|u\|_X$ for all $t\geq0$ and $u\in X$, then the $C_0$-semigroup is said to be uniformly bounded.

From now on, and for the rest of the paper, we will only consider uniformly bounded $C_0$-semigroups.

The infinitesimal generator of $\{S_t\}_{t \geq 0}$ is the linear operator $L$ defined as
$$Lu=\lim_{t\to0}\frac{S_tu-u}{t}$$
with domain $D(L)=\{u\in X:Lu~\hbox{exists}\}\subset X$. In this case,
we write
$$S_t\equiv e^{tL}.$$
If $u\in D(L)$ then the $X$-valued function $v=e^{tL}u$ is differentiable
for $t\geq0$ and satisfies the equation
$\partial_tv = Lv$ for $t>0$, with $v=u$ at $t=0$.

The infinitesimal generator $L$ is a closed operator with domain $D(L)$ dense in $X$.
Conversely, a linear operator $(L,D(L))$ on $X$ is said to generate a semigroup if there is 
a semigroup for which $L$ is its infinitesimal generator.
It follows form the Hille--Yosida and the Lumer--Phillips theorems that
a linear operator $(L,D(L))$ on $X$ is the infinitesimal generator of a $C_0$-semigroup
$\{e^{tL}\}_{t\geq0}$ satisfying $\|e^{tL}\|\leq M$, for some $M\geq1$ and all $t\geq0$, if and only if
$L$ is closed, $D(L)$ is dense in $X$, $(-\infty,0)\subset\rho(-L)$ and
$$\sup_{\mu>0}\|\mu^n(\mu I-L)^{-n}\|\leq M\qquad\hbox{for}~n\geq1.$$

If $L$ is the infinitesimal generator of a uniformly bounded $C_0$-semigroup $\{e^{tL}\}_{t\geq0}$ then
the last statement implies that $A=-L$ is a nonnegative operator.
Thus, the fractional powers $A^s=(-L)^s$ can be defined for any $s>0$ as in the previous subsection.
We will need the following result, proved in \cite[Lemma~6.1.5]{Martinez-book}
when $0<\alpha<1$. The generalization to $\alpha>1$ is obtained by induction using Balakrishnan's operators.

\begin{thm}\label{lem:inverse}
Let $L$ be the infinitesimal generator of a uniformly bounded $C_0$-semigroup $\{e^{tL}\}_{t\geq0}$ on $X$. Then,
for any $\varepsilon>0$ and $\alpha>0$,
$$(\varepsilon I-L)^{-\alpha}u = \frac{1}{\Gamma(\alpha)} \int^\infty_0 e^{-\varepsilon t} e^{tL}u\, \frac{dt}{t^{1-\alpha}}$$
where $\Gamma$ denotes the Gamma function. If $0\in\rho(L)$ then this formula is also valid for $\varepsilon=0$.
\end{thm}

\section{Basic properties of the subordination formula solution}\label{sec:extension}

In this section, $L$ is the infinitesimal generator of a uniformly bounded
$C_0$-semigroup $\{e^{tL}\}_{t\geq0}$ such that $\|e^{tL}u\|_X\leq M\|u\|_X$, for all $u\in X$
and $t\geq0$, for some $M\geq1$.

\begin{lem}\label{lem:extension_equation}
Fix any $s>0$ noninteger. Given any $u\in X$, define $U(y)=U[u](y)$, $y>0$, as in \eqref{eq:Usemigroup}.
The following properties hold.
\begin{enumerate}[$(1)$]
\item The integral defining $U(y)$ is absolutely convergent in the sense of Bochner and
$$\|U(y)\|_X\leq M\|u\|_X\qquad\hbox{for all}~y>0.$$
\item $U\in C([0,\infty);X)$ and $\lim_{y\to0}U(y)=u$ in $X$.
\item For each $y>0$ and $k\geq0$, $U(y)\in D(L^k)$ and
$$L^kU(y)=\frac{(-1)^k}{4^s\Gamma(s)}\int_0^\infty\big(\tfrac{1-2s}{y}\partial_y+\partial_{yy}\big)^k\big(y^{2s}e^{-y^2/(4t)}\big)
e^{tL}u\,\frac{dt}{t^{1+s}}.$$
\item For any $k\geq0$, $U\in C^\infty((0,\infty);D(L^k))$ and, for any $m\geq1$,
$$\frac{d^m}{dy^m}U(y)=\frac{1}{4^s\Gamma(s)}\int_0^\infty\frac{\partial^m}{\partial y^m}\big(y^{2s}e^{-y^2/(4t)}\big)e^{tL}u\,\frac{dt}{t^{1+s}}.$$
In particular,
$$\big(\tfrac{1-2s}{y}\partial_y+\partial_{yy}\big)^kU=(-L)^kU(y)\qquad\hbox{for all}~y>0.$$
\end{enumerate}
\end{lem}

\begin{proof}
The result was proved for $0<s<1$ in \cite{Gale-Miana-Stinga}, see also \cite{Stinga-Torrea} for the case of
Hilbert spaces.
A careful inspection shows that the proofs in \cite{Gale-Miana-Stinga} extend to any noninteger $s>0$
without major modifications. We only sketch the argument for $(3)$, which uses an integration by parts
idea of Miana \cite{Miana}. Let us define
$$v(t)=\int_0^te^{rL}u\,dr\qquad t>0.$$
Then $v(t)\in D(L)$ and $\partial_tv(t) = e^{t L} u$. Integration by parts gives,
\begin{align*}
    U(y) &= \frac{y^{2s}}{4^s\Gamma(s)}\int_0^\infty e^{-y^2/(4t)}\partial_tv(t)\,\frac{dt}{t^{1+s}} \\
    &=-\frac{y^{2s}}{4^s\Gamma(s)}\int_0^\infty \partial_t \bigg( \frac{e^{-y^2/(4t)}}{t^{1+s}} \bigg)v(t)\,dt \\
    &= -\frac{1}{4^s\Gamma(s)}\int_0^\infty\big(\tfrac{1-2s}{y}\partial_y+\partial_{yy}\big)\big(y^{2s}e^{-y^2/(4t)}\big)
v(t)\,\frac{dt}{t^{1+s}}.
\end{align*}
Using that $Lv(t)=e^{tL}u-u$ and the dominated convergence theorem, it can be verified that $U(y)\in D(L)$ and
\begin{align*}
    LU(y) &= -\frac{1}{4^s\Gamma(s)}\int_0^\infty \big(\tfrac{1-2s}{y}\partial_y+\partial_{yy}\big)\big(y^{2s}e^{-y^2/(4t)}\big)\big(e^{tL}u-u\big)\,\frac{dt}{t^{1+s}}\\
&= -\frac{1}{4^s\Gamma(s)}\int_0^\infty \big(\tfrac{1-2s}{y}\partial_y+\partial_{yy}\big)\big(y^{2s}e^{-y^2/(4t)}\big) e^{tL}u\,\frac{dt}{t^{1+s}}.
\end{align*}
In the last line we used that
$$\frac{1}{4^s\Gamma(s)}\int_0^\infty y^{2s}e^{-y^2/(4t)}\,\frac{dt}{t^{1+s}} =1\qquad y>0$$
so its derivative with respect to $y$ vanishes. This shows $(3)$ for the case $k=1$. The general case $k\geq2$ is proved by induction using
the same integration by parts strategy.
\end{proof}

The next formula for $U$ in terms of $(-L)^su$ was derived for $0<s<1$ and $u\in D(L)$
in \cite{Gale-Miana-Stinga}. Inspecting the proof in \cite{Gale-Miana-Stinga}, we see that,
because Theorem \ref{lem:approx_identity} holds for $\alpha=s$ and $u\in D((-L)^s)$, we can extend it to any $s>0$ noninteger.
Thus, the proof is omitted.

\begin{thm}\label{lem: extension_epsilon_presentation}
Suppose that $0 \in \rho(L)$. Let $s>0$ be noninteger, $f\in X$ and assume that $u\in D((-L)^s)$ is the
unique solution to $(-L)^su=f$. Then $U(y)=U[u](y)$ given in \eqref{eq:Usemigroup} can also be written as
$$U(y)=\lim_{\varepsilon\to0}\frac{1}{\Gamma(s)}\int_0^\infty e^{-\varepsilon t}e^{-y^2/(4t)}e^{tL}f\,\frac{dt}{t^{1-s}}
\qquad\hbox{in}~X.$$
\end{thm}

\section{Uniqueness}\label{sec:uniqueness}

The main result of this section is the following.

\begin{thm}[Uniqueness for any $s>0$ noninteger]\label{thm:uniquenessgeneral}
Let $L$ be the infinitesimal generator of a uniformly bounded $C_0$-semigroup $\{e^{tL}\}_{t\geq0}$ on $X$.
Assume that $0\in\rho(L)$. Fix $s>0$ noninteger. Then the $(2[s]+2)$-order initial value extension problem
\begin{equation}\label{eq:problem1general}
\begin{cases}
\big(L+\frac{1-2(s-[s])}{y}\partial_y+\partial_{yy}\big)^{[s]+1}U=0&\hbox{for}~y>0\\
\lim_{y\to0}\big(L+\frac{1-2(s-[s])}{y}\partial_y+\partial_{yy}\big)^m U =u_m&\hbox{in}~X \\
\lim_{y\to0}y^{1-2(s-[s])}\partial_y\big(L+\frac{1-2(s-[s])}{y}\partial_y+\partial_{yy}\big)^mU=v_m&\hbox{in}~X
\end{cases}
\end{equation}
and the second order extension problem with $(2[s]+2)$ initial conditions
\begin{equation}\label{eq:problem2general}
\begin{cases}
LU+\frac{1-2s}{y}\partial_yU+\partial_{yy}U=0&\hbox{for}~y>0\\
\lim_{y\to0}\big(\frac{2}{y}\partial_y\big)^m U =u_m&\hbox{in}~X \\
\lim_{y\to0}y^{1-2(s-[s])}\partial_y\big(\frac{2}{y}\partial_y\big)^mU=v_m&\hbox{in}~X
\end{cases}
\end{equation}
where $u_m,v_m\in X$, $0\leq m\leq [s]$, have at most one classical solution.
\end{thm}

We start by showing that uniqueness in \eqref{eq:problem2general} fails for $s>1$ if we only impose two initial conditions
in terms of $U$ and $\partial_yU$. For this, we begin by  finding two independent solutions to the following Bessel ODE:
\begin{equation}\label{eq:BesselODE}
\phi''(y)+ \frac{a}{y}\phi'(y)=\lambda\phi(y)\qquad \hbox{for}~y,\lambda>0.
\end{equation}
In the extension problem, $a=1-2s$, so if $s>0$ then $a\in(-\infty,1)$. For this analysis,
we allow any $a<1$, $a\neq0$, and study when we can obtain existence and uniqueness of 
solutions to a typical initial value problem for \eqref{eq:BesselODE} where we prescribe $\phi$
and  $\phi'$ at $y=0$.

Since $y=0$ is a regular singular point for \eqref{eq:BesselODE}, we can find solutions
by applying the classical method of power series expansions as in \cite[Chapter~5]{Simmons}.
The indicial equation is $\rho(\rho-1)+\rho a=0$,
which has roots $\rho_1=0$ and $\rho_2=1-a$. We look for solutions of the form
$\phi_i(y,\lambda)=y^{\rho_i}\sum_{k=0}^\infty a_{i,k}y^k$, $i=1,2$. Inserting these into \eqref{eq:BesselODE} one can find
the recurrence relations for the coefficients $a_{i,k}$ and get that
$$\phi_1(y,\lambda)=\Gamma(1-\tfrac{1-a}{2})\sum_{k=0}^\infty\frac{y^{2k}\lambda^k}{2^{2k}\Gamma(k+1)\Gamma(k+1-\frac{1-a}{2})}$$
and
$$\phi_2(y,\lambda)=\frac{\Gamma(\tfrac{1-a}{2})}{2}y^{1-a}\sum_{k=0}^\infty\frac{y^{2k}\lambda^k}{2^{2k}\Gamma(k+1)\Gamma(k+1+\frac{1-a}{2})}$$
are two independent solutions. By taking into account the power
series expansions of the Bessel functions $I_{\pm\nu}$ (see \cite[p.~108]{Lebedev}),
$$\phi_1(y,\lambda)=\Gamma(1-\tfrac{1+a}{2})(\lambda^{1/2} y/2)^{\frac{1-a}{2}}I_{-\frac{1-a}{2}}(\lambda^{1/2} y)$$
and
$$\phi_2(y,\lambda)=\frac{\Gamma(\tfrac{1-a}{2})}{2}(2\lambda^{-1/2}y)^{\frac{1-a}{2}} I_{\frac{1-a}{2}}(\lambda^{1/2} y).$$
Next, using the power series expansions above, it follows that, as $y\to0$,
\begin{equation}\label{eq:boundaryconditionsBessel}
\begin{cases}
\phi_1(y,\lambda)\to1,&\partial_y\phi_1(y,\lambda)\sim-\frac{\lambda}{a}y,\\
\phi_2(y,\lambda)\sim\frac{1}{1-a}y^{1-a},&\partial_y\phi_2(y,\lambda)\sim y^{-a}.
\end{cases}
\end{equation}

\begin{lem}\label{lem:Besselivp}
Fix $\lambda>0$, $a<1$ and $b\in\R$. The initial value problem
$$\begin{cases}
\phi''(y)+ \frac{a}{y}\phi'(y)=\lambda\phi(y)&\hbox{for}~y>0\\
\lim_{y\to0}\phi(y)=\alpha\\
\lim_{y\to0}y^b\phi'(y)=\beta
\end{cases}$$
has at most one classical solution for arbitrary $\alpha,\beta\in\R$ if and only if $a=b\in[-1,1)$.
\end{lem}

\begin{proof}
Suppose that $a\neq0$. The general solution is $\phi=c_1\phi_1+c_2\phi_2$, where $c_1,c_2\in\R$. Since $\phi(0)=c_1$
and $1-a>0$, it follows from \eqref{eq:boundaryconditionsBessel} that $\phi=\alpha\phi_1+c_2\phi_2$, where $c_2$ is arbitrary.
Next, using again \eqref{eq:boundaryconditionsBessel},
$$y^b\phi'(y)\sim-\frac{\alpha\lambda}{a}y^{1+b}+c_2y^{b-a}\qquad\hbox{as}~y\to0.$$
If $b<a$ this limit blows up unless $c_2=0$, but then we are forced to take $\beta=0$ if $b>-1$ or $\beta=-\frac{\alpha\lambda}{a}$ if $b=-1$.
If $b>a$, the second term vanishes in the limit independently of $c_2$, so there is no uniqueness for all $\alpha$ and $\beta$.
If $b=a$, the first term vanishes if and only if $-1<a<1$ and we can choose $c_2=\beta$ to get the unique solution.
If $b=a={-1}$ then we choose $c_2=\beta-\alpha\lambda$ and uniqueness follows.

When $a=0$, one can use the independent solutions $e^{y\lambda^{1/2}}$ and $e^{-y\lambda^{1/2}}$.
\end{proof}

Recall that $s>0$ is noninteger. Therefore, we focus our attention to \eqref{eq:BesselODE} for $-1<a<1$.
For the nonhomogeneous problem
\begin{equation}\label{eq:BesselODEg}
\phi''(y)+ \frac{a}{y} \phi'(y) = \lambda\phi(y)+g(y)\qquad y,\lambda>0
\end{equation}
where now $a\in(-1,1)$, we apply the method of variation of parameters.
Using the Wroskian for Bessel functions \cite{Lebedev}, a particular solution $\phi_p$ can thus be found as
\begin{equation}\label{eq:formulaforphip}
\phi_p(y,\lambda) =\int^y_0\big(\phi_2(y,\lambda)\phi_1(t,\lambda)-\phi_1(y,\lambda)\phi_2(t,\lambda)\big)g(t)t^a\,dt.
\end{equation}
It is easy to verify that
\begin{equation}\label{eq:Neumannphip}
\lim_{y\to0}\phi_p(y,\lambda)=\lim_{y\to0}y^a\phi_p'(y,\lambda)=0.
\end{equation}

Since the power series representations of $\phi_1$ and $\phi_2$ have infinite radius of convergence,
we can replace $\lambda$ by any bounded linear operator $T$ on $X$ to get well-defined $X$-valued functions of $y$:
$$\phi_1(y,T)u=\Gamma(1-\tfrac{1-a}{2})\sum_{k=0}^\infty\frac{y^{2k}}{2^{2k}\Gamma(k+1)\Gamma(k+1-\frac{1-a}{2})}T^ku$$
and
$$\phi_2(y,T)v=\frac{\Gamma(\tfrac{1-a}{2})}{2}y^{1-a}\sum_{k=0}^\infty\frac{y^{2k}}{2^{2k}\Gamma(k+1)\Gamma(k+1+\frac{1-a}{2})}T^kv.$$
Clearly, these are classical solutions to the $X$-valued Bessel equation \eqref{eq:BesselODE} with $\lambda=T$ and
with corresponding boundary conditions (see \eqref{eq:boundaryconditionsBessel})
$$\begin{cases}
    \lim_{y \to 0}\phi_1(y,T)u = u,& \lim_{y \to 0} y^{a}\phi'_1(y,T)u = 0, \\
    \lim_{y \to 0}\phi_2(y,T)v = 0, & \lim_{y \to 0} y^{a}\phi'_2(y,T)v = v.
\end{cases}$$
Similarly, defining $\phi_p(y,T)u$ by replacing $\lambda$ by $T$ in the formula for $\phi_p$ in \eqref{eq:formulaforphip},
one can verify that $\phi_p(y,T)u$ satisfies \eqref{eq:BesselODEg} and \eqref{eq:Neumannphip} with $T$ in place of $\lambda$.

\begin{lem}\label{lem:unique_A_bounded_nonhomogenous}
Let $T:X\to X$ be a bounded linear operator and fix $f\in C([0,\infty);X)$.
Let $a\in(-1,1)$. Then the problem 
$$\begin{cases}
\partial_{yy}U+\frac{a}{y}\partial_yU=TU+f&\hbox{for}~y>0 \\
\lim_{y\to0}U(y)=u&\hbox{in}~X \\
\lim_{y \to 0} y^a\partial_y U(y)=v&\hbox{in}~X
\end{cases}$$
for $u,v\in X$, has a unique classical solution. If $u=v=0$, the unique solution is
$$U(y)=\int^y_0\big(\phi_2(y,T)\phi_1(t,T)-\phi_1(y,T)\phi_2(t,T)\big)f(t) t^a\,dt.$$
\end{lem}

\begin{proof}
For uniqueness, we use Gronwall's inequality as in \cite{Meichsner-Seifert}. By linearity,
it is enough to assume that $u=v=f=0$. Notice that the equation for $U$ can be written as
$$\partial_y\big(y^a\partial_yU\big)=y^{a}TU.$$
Integrating this twice, using the initial conditions and Fubini's theorem, we get
$$U(y)=\frac{1}{1-a}\int_0^y\big(y^{1-a}-t^{1-a}\big)TU(t)t^a\,dt.$$
Since $T$ is bounded,
$$\|U(y)\|_X\leq\frac{\|T\|}{1-a}\int_0^y\big(y^{1-a}-t^{1-a}\big)\|U(t)\|_Xt^a\,dt$$
and Gronwall's inequality implies that $U=0$.
\end{proof}

Clearly, the proof of Lemma \ref{lem:unique_A_bounded_nonhomogenous} fails when $T$ is unbounded.
However, the result is still true for $-1<a<1$ in the unbounded case (see Theorem \ref{lem:unique_A_unbounded})
as was shown in \cite{Stinga-Torrea} for Hilbert spaces and in \cite{Meichsner-Seifert} for Banach spaces.
We next revisit the proof of the general case of \cite{Meichsner-Seifert} in our context
and provide some of the missing details where necessary.

To this end, we need to introduce sectorial operators and their functional calculus, see \cite{Hasse,Martinez-book}.
A sector in the complex plane of angle $\omega\in(0,\pi]$ is defined as
$S_\omega= \{z \in \mathbb{C}: z \neq 0, |\arg z|< w \}$, and we set $S_0= (0,\infty)$.
A linear operator $(A,D(A))$ on $X$ is called sectorial if there is an angle $\omega\in[0,\pi)$ such that
$\sigma(A) \subseteq S_\omega$ and 
$\sup\{\|\lambda(\lambda I - A)^{-1}\| : \lambda \in \C \setminus \overline{S_{\omega'}} \} < \infty$, for all $\omega' \in (\omega,\pi)$.
For such an operator $A$, $\omega_A = \min \{0 \leq\omega < \pi : A \in S_\omega \}$ is
the angle of sectoriality of $A$. 

It is well known that every closed nonnegative operator is a sectorial operator and any sectorial operator is nonnegative.
If $L$ is the generator of a uniformly bounded $C_0$-semigroup then $A=-L$ is sectorial, with $\omega_A\leq\pi/2$.

We say that a meromorphic function $F$ on a sector $S_\omega$ has a
finite polynomial limit $c\in\C$ at $z=0$ if $F(z)-c=O(|z|^\alpha)$, as $z\to0$, for some $\alpha>0$.
Similarly, $F$ is said to have a finite polynomial limit $d\in\C$ at $\infty$ if $F(z^{-1})$ has polynomial limit $d$ at $0$. 

Let $A$ be a sectorial operator with angle of sectoriality $\omega_A$.
For bounded, holomorphic functions $F$ on an open sector containing $S_{\omega_A}$ that have finite polynomial limits both 
at $0$ and at $\infty$, there is a well defined primary functional calculus,
namely, the operator $F(A)$ is defined via a Cauchy integral over an
open sector containing $S_{\omega_A}$, see \cite[Lemma 2.3.2]{Hasse}.
 
A sequence $\{A_n\}_{n\geq1}$ of sectorial operators is called uniformly sectorial
with angle $\omega\in[0,\pi)$ if $A_n$ is sectorial with angle of sectoriality $\omega$
and $\sup_n\sup\{\|\lambda(\lambda I-A)^{-1}\|:\lambda\in\C\setminus\overline{S_{\omega'}}\}<\infty$
for every $\omega'\in(\omega,\pi)$. A uniformly sectorial sequence $\{A_n\}_{n\geq1}$ with angle $\omega$
is a sectorial approximation of a sectorial operator $A$ with the same angle of sectoriality
if $\lambda\in\rho(A)$ and $(\lambda I-A_n)^{-1}\to(\lambda I-A)^{-1}$ in the operator norm,
for some $\lambda\notin\overline{S_\omega}$. In this case, this is in fact true for all $\lambda\notin\overline{S_\omega}$.
Moreover, $F(A_n) \to F(A)$ in operator norm,
whenever $F(A)$ is defined by the primary functional calculus for $A$, see \cite[Lemma 2.6.7]{Hasse}.
 
\begin{thm}[Uniqueness for $a\in(-1,1)$]\label{lem:unique_A_unbounded}
Let $L$ be the infinitesimal generator of a uniformly bounded $C_0$-semigroup $\{e^{tL}\}_{t\geq0}$ on  $X$.
Assume that $0\in\rho(L)$. Let $a\in(-1,1)$.
Then the second order initial value problem
$$\begin{cases}
LU+\frac{a}{y}\partial_y U+\partial_{yy}U=0&\hbox{for}~y>0\\
\lim_{y\to0}U(y)=u&\hbox{in}~X \\
\lim_{y \to0}y^a\partial_yU(y)=v&\hbox{in}~X,
\end{cases}$$
for $u,v\in X$, has at most one classical solution. 
 \end{thm}

\begin{proof}
By linearity, we can assume that $u=v=0$.
Since $L$ is injective, it is enough to prove that $(-L)^{-1}U=0$. Now $(-L)^{-1}U$ satisfies the same equation as $U$, therefore
we may assume that $U\in C((0,\infty);D(L^2)) \cap C([0, \infty); D(L)) \cap C^2((0,\infty);D(L))$.

As $L$ is the infinitesimal generator of a uniformly bounded $C_0$-semigroup,
the operator $A=-L$ is closed, nonnegative and has dense domain $\overline{D(A)}=X$.
For any $\vep>0$, define
$$A_\varepsilon = A(I+\varepsilon A)^{-1}.$$
Since $A$ is nonnegative, it can be seen that $\{A_\vep\}_{\vep>0}$ is a family of bounded,
nonnegative operators, with nonnegativity constant uniformly bounded in $\vep$. Moreover, $0\in\rho(A_\vep)$.
Furthermore, by using \cite[Proposition~2.1.1(f)]{Hasse}, $\{A_\vep\}_{\vep>0}$ is uniformly sectorial.
Finally, it can be proved that $A_\vep$ is a sectorial approximation of $A$.

Next, we define,
$$U_\varepsilon=(I+\varepsilon A)^{-1}U.$$
Since $A$ is nonnegative and $D(A)$ is dense in $X$,
\begin{equation}\label{eq:UepstoU}
\lim_{\vep\to0}U_\vep(y)=U(y)\qquad\hbox{for all}~y\geq0.
\end{equation}
Moreover, $U_\vep$ is a solution to
\begin{equation}\label{eq:regularizedproblem}
\begin{cases}
\partial_{yy}U_\varepsilon+\frac{a}{y}\partial_y U_\varepsilon=A_\varepsilon U_\varepsilon+f_\varepsilon &\hbox{for}~y>0 \\
\lim_{y\to0}U_\varepsilon(y)=0&\hbox{in}~X \\
\lim_{y \to 0}y^a\partial_y U_\varepsilon(y)= 0&\hbox{in}~X
\end{cases}
\end{equation}
where 
$$f_\varepsilon(y)=AU_\varepsilon(y)-A_\varepsilon U_\varepsilon(y).$$
In view of \eqref{eq:UepstoU}, it is enough to prove that, for each $y>0$, $U_\vep(y)\to0$ as $\vep\to0$.

Since $A_\vep$ is a bounded linear operator on $X$ and $f_\vep\in C([0,\infty);X)$, Lemma \ref{lem:unique_A_bounded_nonhomogenous}
implies that the unique classical solution to \eqref{eq:regularizedproblem} is
$$U_\vep(y)=\int^y_0\big(\phi_2(y,A_\vep)\phi_1(t,A_\vep)-\phi_1(y,A_\vep)\phi_2(t,A_\vep)\big)f_\vep(t)t^a\,dt\qquad y>0.$$
Recall that $U\in D(A)$. Then $A(I+\vep A)^{-1}U=(I+\vep A)^{-1}AU$ and we can write
\begin{equation}\label{eq:f_epsilon}
f_\varepsilon= \varepsilon^{-1}(\varepsilon^{-1}+A)^{-1}A(\varepsilon^{-1} +A)^{-1}  (AU)
\end{equation}
From here, using again that $A$ is nonnegative, we conclude that $f_\vep(y)\to 0$, for each $y\geq0$.
Furthermore, since $AU \in C([0,\infty);X)$, and the bounded operators 
$\varepsilon^{-1}(\varepsilon^{-1}+A)^{-1}$ and $A(\varepsilon^{-1} +A)^{-1}$ have norms uniformly bounded in $\vep>0$,
it follows that $f_\varepsilon(y)$ is uniformly bounded in $y\in[0,y_0]$ and $\vep\geq0$.

At this point, we cannot take the limit in \eqref{eq:UepstoU} inside the integral defining $U_\vep$
because, say, $\phi_1(y,A)$ is not well defined
unless stronger conditions on $U$ are imposed. Instead, we consider $e^{-y_0A^{1/2}}U_\vep(y)$, where $y_0>0$ is fixed
and $\{e^{-tA^{1/2}}\}_{t\geq0}$ is the Poisson semigroup associated to $A$ (see \cite{Martinez-book,Yosida}), and prove that
\begin{equation}\label{eq:uniquelimitepsilon}
e^{-y_0A^{1/2}}U(y)=\lim_{\vep\to0}e^{-y_0A_\vep^{1/2}}U_\vep(y)=0\qquad\hbox{for all}~0<y<y_0.
\end{equation}
Thus, by the injectivity of $e^{-y_0A^{1/2}}$, the desired conclusion $U=0$ follows.

To this end, fix $0<y<y_0$ and define the holomorphic functions
$$F_1(y,z)=e^{-y_0z}\phi_1(y,z^2)=e^{-y_0z}\Gamma(1-\tfrac{1-a}{2})
\sum^{\infty}_{k=0}\frac{y^{2k}z^{2k}}{2^{2k}\Gamma(k+1)\Gamma(k+1-\frac{1-a}{2})}$$
and
$$F_2(y,z)=e^{-y_0z}\phi_2(y,z^2)=e^{-y_0z}\frac{\Gamma(\frac{1-a}{2})}{2}y^{1-a}
\sum^{\infty}_{k=0} \frac{y^{2k}z^{2k}}{2^{2k}\Gamma(k+1)\Gamma(k+1+\frac{1-a}{2})}.$$
Then we can write
$$e^{-y_0A_\vep^{1/2}}U_\vep(y)=\int^y_0 \big(F_2(y,A^{1/2}_\vep)F_1(t,A^{1/2}_\vep)-
F_1(y,A^{1/2}_\vep)F_2(t,A^{1/2}_\vep)\big)f_\vep(t)t^a\,dt.$$
Using the power series representations (or, equivalently, the asymptotic expansions of Bessel functions \cite{Lebedev}),
it is not difficult to check that $F_1(y,z)$ and $F_2(y,z)$ are holomorphic and bounded on proper subsectors of $\Re(z)>0$,
and have finite polynomial limits when approaching $0$ and $\infty$ from within $\Re(z)>0$. If we prove that $A_\vep^{1/2}$ is a sectorial approximation 
of angle $\omega\leq\pi/4$ (note that $\overline{S_{\pi/4}}$ is still a proper subsector of $\Re(z)>0$) of $A^{1/2}$ then, by the primary functional calculus,
$$F_1(y,A^{1/2}_\vep)\to F_1(y,A^{1/2})\qquad\hbox{and}\qquad F_2(y,A^{1/2}_\vep)\to F_2(y,A^{1/2})$$
as $\vep\to0$, where both $F_1(y,A^{1/2})$ and $F_2(y,A^{1/2})$ are now bounded operators.
Therefore, the dominated convergence theorem can be applied and \eqref{eq:uniquelimitepsilon} follows.

We are left to show that $A_\vep^{1/2}$ is a sectorial approximation of $A^{1/2}$.
We want to prove that 
$(\lambda I + A^{1/2}_\vep)^{-1} \to (\lambda I+ A^{1/2})^{-1}$ given that
$(\lambda I + A_\vep)^{-1} \to (\lambda I + A)^{-1}$.
By \cite[Proposition 5.3.2 and Theorem 5.4.1]{Martinez-book}, $A^{1/2}$ is a sectorial operator with angle of sectoriality $\omega_A/2$,
where $\omega_A$ is the angle of sectoriality of $A$.
Since $A_\vep$ is a nonnegative operator with nonnegativity constant uniform in $\vep$, so does $A^{1/2}_\vep$.
Hence  $(\lambda I+A^{1/2}_\vep)^{-1}$ exists for any $\lambda>0$ and
$$(\lambda I+ A^{1/2}_\vep)^{-1} = \lambda^{-1}(A^{-1/2}_\vep+\lambda^{-1}I)^{-1} A^{-1/2}_\vep.$$
Next, we see that $(A^{-1/2}_\vep+\lambda^{-1}I) \to (A^{-1/2}+\lambda^{-1}I)$
in norm as $\vep \to 0$, so that $(A^{-1/2}_\vep+\lambda^{-1}I)^{-1} \to (A^{-1/2}+\lambda^{-1}I)^{-1}$ as $\vep \to 0$.
Combining all together, we have that $(\lambda I+ A^{1/2}_\vep)^{-1} \to (\lambda I+ A^{1/2})^{-1}$
in norm. Therefore, $\{A^{1/2}_\vep\}_{\vep>0}$ is a sectorial approximation of $A^{1/2}$.
\end{proof}

Before proving Theorem \ref{thm:uniquenessgeneral}, will need the following differential identities lemma.

\begin{lem}\label{lem:recursion}
Let $L$ be a linear operator on $X$ and $a\in\R$. The differential identity
$$\big(L+\tfrac{a}{y}\partial_y+\partial_{yy}\big)\big(\tfrac{2}{y}\partial_y\big)=
\big(\tfrac{2}{y}\partial_y\big)\big(L+\tfrac{a-2}{y}\partial_y+\partial_{yy}\big)$$
holds. Moreover, if $U$ is a sufficiently smooth function that satisfies the equation
$$LU+\tfrac{a}{y}\partial_yU+\partial_{yy}U=0$$
then, for any $m\geq0$,
$$\big(L+\tfrac{a+2m}{y}\partial_y+\partial_{yy}\big)^{m+1}U=0$$
and, for any $0\leq m\leq n$,
$$\big(L+\tfrac{a+2n}{y}\partial_y+\partial_{yy}\big)^{m}U=\frac{n!}{(n-m)!}\big(\tfrac{2}{y}\partial_y\big)^{m}U.$$
\end{lem}

\begin{proof}
We can write
$$\big(L+\tfrac{a}{y}\partial_y+\partial_{yy}\big)\big(\tfrac{2}{y}\partial_y\big)=
\big(\tfrac{2}{y}\partial_y\big)L+\big(\tfrac{a}{y}\partial_y+\partial_{yy}\big)\big(\tfrac{2}{y}\partial_y\big)$$
and it is easy to check that
$$\big(\tfrac{a}{y}\partial_y + \partial_{yy} \big) \big(\tfrac{2}{y}\partial_y\big)
=\big(\tfrac{2}{y}\partial_y\big)\big(\tfrac{a-2}{y} \partial_y + \partial_{yy} \big).$$

Let $U$ be as in the statement. For any $m\geq0$, by the differential identity,
\begin{align*}
    \big(L+\tfrac{a+2m}{y}\partial_y+\partial_{yy}\big)^{m+1}U &=
    \big(L+\tfrac{a+2m}{y}\partial_y+\partial_{yy}\big)^m\big[\big(L+\tfrac{a}{y}\partial_y+\partial_{yy}\big)U+\big(\tfrac{2m}{y}\partial_y\big)U\big] \\
    &=\big(L+\tfrac{a+2m}{y}\partial_y+\partial_{yy}\big)^m\big(\tfrac{2m}{y} \partial_y \big) U \\
    &= \big(\tfrac{2m}{y} \partial_y\big)\big(L+\tfrac{a+2m-2}{y}\partial_y+\partial_{yy}\big)^mU \\
    &= m\big(\tfrac{2}{y} \partial_y\big)\big(L+\tfrac{a+2m-2}{y}\partial_y+\partial_{yy}\big)^{m-1}\big(\tfrac{2m-2}{y}\partial_y\big)U \\
    &= m(m-1)\big(\tfrac{2}{y} \partial_y\big)^2\big(L+\tfrac{a+2m-4}{y}\partial_y+\partial_{yy}\big)^{m-1}U.
\end{align*}
Continuing this process $m$-times,
$$\big(L+\tfrac{a+2m}{y}\partial_y+\partial_{yy}\big)^{m+1}U=m!\big(\tfrac{2}{y} \partial_y\big)^m\big(L+\tfrac{a}{y}\partial_y+\partial_{yy}\big)U=0.$$
The statement for $0\leq n\leq m$ is proved in a similar way:
\begin{align*}
    \big(L+\tfrac{a+2n}{y}\partial_y+\partial_{yy}\big)^mU 
    &=\big(L+\tfrac{a+2n}{y}\partial_y+\partial_{yy}\big)^{m-1}\big(\tfrac{2n}{y}\partial_y\big)U \\
    &=n\big(L+\tfrac{a+2n}{y}\partial_y+\partial_{yy}\big)^{m-2}\big(L+\tfrac{a+2n}{y}\partial_y+\partial_{yy}\big)\big(\tfrac{2}{y}\partial_y\big)U \\
    &=n\big(L+\tfrac{a+2n}{y}\partial_y+\partial_{yy}\big)^{m-2}\big(\tfrac{2}{y}\partial_y\big)\big(L+\tfrac{a+2n-2}{y}\partial_y+\partial_{yy}\big)U \\
    &=n(n-1)\big(L+\tfrac{a+2n}{y}\partial_y+\partial_{yy}\big)^{m-2}\big(\tfrac{2}{y}\partial_y\big)^2U.
\end{align*}
Continuing the process for $m$-times,
\begin{align*}
\big(L+\tfrac{a+2n}{y}\partial_y+\partial_{yy}\big)^mU &= n(n-1)(n-2)\cdots (n-m+1)\big(\tfrac{2}{y} \partial_y\big)^mU \\
&=\frac{n!}{(n-m)!}\big(\tfrac{2}{y} \partial_y\big)^mU.
\end{align*}
\end{proof}

For the reasons mentioned at the end of the introduction, we believe that stating
and proving Theorem \ref{thm:uniquenessgeneral} for the case when $1<s<2$ will be useful and instructive.

\begin{thm}[Uniqueness for $1<s<2$]\label{cor:unique1s2}
Let $L$ be the infinitesimal generator of a uniformly bounded $C_0$-semigroup $\{e^{tL}\}_{t\geq0}$
on $X$. Assume that $0 \in \rho(L)$. Fix $1<s<2$. Then the fourth order initial value extension problem
\begin{equation}\label{eq:problem1}
\begin{cases}
\big(L+\frac{3-2s}{y}\partial_y+\partial_{yy}\big)^{2}U=0&\hbox{for}~y>0\\
\lim_{y\to0}U(y)=u_1&\hbox{in}~X \\
\lim_{y\to0}y^{3-2s}\partial_yU(y)=u_2&\hbox{in}~X \\
\lim_{y\to0}\big(LU+\frac{3-2s}{y}\partial_yU+\partial_{yy}U\big)= u_3&\hbox{in}~X \\
\lim_{y\to0}y^{3-2s}\partial_y\big(LU+\frac{3-2s}{y}\partial_yU+\partial_{yy}U\big)=u_4&\hbox{in}~X
\end{cases}
\end{equation}
and the second order initial value extension problem
\begin{equation}\label{eq:problem2}
\begin{cases}
LU+\frac{1-2s}{y}\partial_yU+\partial_{yy}U=0&\hbox{for}~y>0\\
\lim_{y\to0}U(y)=u_1&\hbox{in}~X \\
\lim_{y\to0}y^{3-2s}\partial_yU=u_2&\hbox{in}~X \\
\lim_{y\to0}\frac{2}{y}\partial_yU= u_3&\hbox{in}~X \\
\lim_{y\to0}y^{3-2s}\partial_y\big(\frac{2}{y}\partial_yU\big)=u_4&\hbox{in}~X
\end{cases}
\end{equation}
where $u_m\in X$, $1\leq m\leq4$, have at most one classical solution.
 \end{thm}
 
\begin{proof}
By linearity, it is enough to assume that $u_1=u_2=u_3=u_4=0$.

Let $U$ be a classical solution to \eqref{eq:problem1}.
Then $V=LU+\tfrac{3-2s}{y}\partial_yU+\partial_{yy}U$ solves
$$\begin{cases}
LV+\frac{a}{y}\partial_yV+\partial_{yy}V=0&\hbox{for}~y>0\\
\lim_{y\to0}V(y) =0&\hbox{in}~X \\
 \lim_{y\to0}y^{a}\partial_y V = 0&\hbox{in}~X
\end{cases}$$
where $a=3-2s\in(-1,1)$. By Theorem \ref{lem:unique_A_unbounded}, $V(y)=0$ for all $y\geq0$.
Consequently, $U$ solves
$$\begin{cases}
LU+\frac{3-2s}{y}\partial_yU+\partial_{yy}U = 0 &\hbox{for}~y>0\\
\lim_{y\to0}U(y)=0&\hbox{in}~X \\
\lim_{y\to0}y^{3-2s}\partial_yU=0&\hbox{in}~X.
\end{cases}$$
Applying again Theorem \ref{lem:unique_A_unbounded} with $a=3-2s\in(-1,1)$, $U(y) = 0$ for all $y\geq0$.  

Next, suppose that $U$ is a classical solution to \eqref{eq:problem2}. For all $y>0$, by Lemma \ref{lem:recursion}
with $a=1-2s$ and $m=n=1$, $\big(L+\tfrac{3-2s}{y}\partial_y+\partial_{yy}\big)^2U=0$ and 
$LU+\tfrac{3-2s}{y}\partial_yU+\partial_{yy}U=\tfrac{2}{y}\partial_yU$.
Therefore, $U$ solves \eqref{eq:problem1} with $u_m=0$ so, by the first part of this proof, $U=0$.
\end{proof}

\begin{proof}[Proof of Theorem \ref{thm:uniquenessgeneral}]
With an induction argument using the reduction idea of the proof of Theorem \ref{cor:unique1s2}
and applying Theorem \ref{lem:unique_A_unbounded}, uniqueness for \eqref{eq:problem1general} follows.
Similarly as in the proof of Theorem \ref{cor:unique1s2}, Lemma \ref{lem:recursion} gives
that a classical solution to \eqref{eq:problem2general} is also a solution to a problem of the form \eqref{eq:problem1general},
and so it is unique. Details are left to the reader.
\end{proof}

\section{The sharp extension problem for $0<s<1$ and $1<s<2$}\label{sec: extension s small}

In this section we state and prove Theorem \ref{thm:extension-general} for $0<s<1$ and $1<s<2$.
This has a two-fold purpose. First, these are the typical fractional powers used in most applications,
and we believe that presenting them separately will give the reader a better understanding
of the strategy of proof for the general case $s>0$. Second, in both cases we can say something more about
the derivative limits characterizing $(-L)^su$ in terms of difference quotients,
see \eqref{eq:neumanns} and \eqref{eq:U Ls}, that we were not able to generalize to $s>2$.
These latter descriptions are important as they give the correct way of approximating the conormal derivative of $U$
with finite differences, a fundamental aspect to compute $(-L)^su$ numerically using the extension.

\begin{thm}[Extension problem for $0<s<1$]\label{thm: extension frac domain}
Let $L$ be the infinitesimal generator of a uniformly bounded $C_0$-semigroup $\{e^{tL}\}_{t\geq0}$
on $X$. Assume that $0 \in \rho(L)$. For $0<s<1$ and $u\in X$, let $U(y)=U[u](y)$ be its extension as in \eqref{eq:Usemigroup}.
The following statements hold.
\begin{enumerate}[$(a)$]
\item If $u\in D((-L)^s)$ and $f=(-L)^su$ then 
\begin{equation}\label{eq:newidentity}
\begin{aligned}
U(y) &= u+\frac{1}{\Gamma(s)}\int_0^\infty\big(e^{-y^2/(4t)}-1\big)e^{tL}f\,\frac{dt}{t^{1-s}} \\
&= u+\frac{y^{2s}}{4^s\Gamma(s)}\int_0^\infty\big(e^{-r}-1\big)e^{\frac{y^2}{4r}L}f\,\frac{dr}{r^{1+s}}.
\end{aligned}
\end{equation}
\item We have that $u\in D((-L)^s)$ if and only if the limits
$$\lim_{y\to0}y^{1-2s}\partial_yU(y)\qquad\hbox{or}\qquad\lim_{y\to0}\frac{U(y)-u}{y^{2s}}$$
exist in $X$ and, in this case,
\begin{equation}\label{eq:neumanns}
\lim_{y\to0}y^{1-2s}\partial_yU(y)=c_s(-L)^su=2s\lim_{y\to0}\frac{U(y)-u}{y^{2s}}\quad\hbox{in}~X,
\end{equation}
where $c_s=\frac{-\Gamma(1-s)}{4^{s-1/2}\Gamma(s)}$.
\item If $u\in D((-L)^s)$ then its extension $U$ given by \eqref{eq:Usemigroup} is the unique classical solution to the 
second order initial value problem
\begin{equation}\label{eq:extension_problem}
\begin{cases}
LU+\frac{1-2s}{y}\partial_yU+\partial_{yy}U=0&\hbox{for}~y>0\\
\lim_{y\to0}U(y)=u&\hbox{in}~X\\
\lim_{y\to0}y^{1-2s}\partial_yU(y)=c_s(-L)^su&\hbox{in}~X.
\end{cases}
\end{equation}
\end{enumerate}
\end{thm}

\begin{proof}
For $(a)$, we point out that \eqref{eq:newidentity} was proved in \cite{Gale-Miana-Stinga}
for the case $u\in D(L)$. The same idea works if we suppose that $u\in D((-L)^s)$ because
of Theorem \ref{lem:approx_identity}. Indeed, 
let $f=(-L)^su$. By Theorems \ref{lem:approx_identity} and \ref{lem:inverse},
$$u=(-L)^{-s}f=\lim_{\varepsilon\to0}(\varepsilon I-L)^{-s}f=\lim_{\varepsilon\to0}\frac{1}{\Gamma(s)}\int_0^\infty
e^{-\varepsilon t}e^{tL}f\,\frac{dt}{t^{1-s}}.$$
Using this and Theorem \ref{lem: extension_epsilon_presentation},
\begin{align*}
    U(y) -u &= \lim_{\varepsilon \to 0}\frac{1}{\Gamma(s)}\int_0^\infty e^{-\varepsilon t}\big(e^{-y^2/(4t)}-1\big)e^{tL}f\,\frac{dt}{t^{1-s}} \\
    &= \frac{1}{\Gamma(s)}\int_0^\infty\big(e^{-y^2/(4t)}-1\big)e^{tL}f\,\frac{dt}{t^{1-s}},
\end{align*}
where in the last line we applied the dominated convergence theorem. The change
of variables $r=y^2/(4t)$ gives the second formula in \eqref{eq:newidentity}.

Let us consider $(b)$. Suppose first that $u\in D((-L)^s)$ and let $f=(-L)^su$.
By \eqref{eq:newidentity} and the dominated convergence theorem,
 $$\lim_{y\to0}\frac{U(y) - u}{y^{2s}}=\lim_{y\to0}\frac{1}{4^s\Gamma(s)}\int^\infty_0   (e^{-r} - 1) e^{\frac{y^2}{4r}L}f \, \frac{dr}{r^{1+s}}
 =\frac{c_s}{2s}f.$$
Similarly, after differentiating inside the integral sign in the first identity in \eqref{eq:newidentity}
and changing variables, we get
$$\lim_{y\to0}y^{1-2s}\partial_yU(y)=-\lim_{y\to0}\frac{1}{4^{s-1/2}\Gamma(s)}\int_0^\infty e^{-r}e^{\frac{y^2}{4r}L}f\,\frac{dr}{r^{s}}
= c_sf.$$
For the converse in $(b)$, let $U[v](y)$ denote the extension of $v\in X$ as in \eqref{eq:Usemigroup} and define
$$T_sv:=\lim_{y\to0}\frac{U[v](y)-v}{y^{2s}}$$
with domain $D(T_s) =\big\{v\in X:T_sv~\hbox{exists}\big\}$.
We just proved that $D((-L)^s)\subset D(T_s)$. To show the opposite inclusion,
let $u\in D(T_s)$ and set $g=T_su\in X$.
Since $(-L)^{-s}$ is a bounded linear operator that commutes with $e^{tL}$,
it is clear that $(-L)^{-s}U[u](y)=U[(-L)^{-s}u](y)$ and
$$(-L)^{-s}g= (-L)^{-s}[T_su]=T_s[(-L)^{-s}u].$$
But now, by the direct statement we just proved applied to $v:=(-L)^{-s}u\in D((-L)^s)$,
$$T_s[(-L)^{-s}u]=T_s[v]=\frac{c_s}{2s}(-L)^sv=\frac{c_s}{2s}(-L)^s[(-L)^{-s}u]=\frac{c_s}{2s}u.$$
Hence, $(-L)^{-s}g=\frac{c_s}{2s}u$ and $u\in D((-L)^s)$, as desired.
A similar argument using the identity $(-L)^{-s}\big(y^{1-2s}\partial_yU[u](y)\big)=y^{1-2s}\partial_yU[(-L)^{-s} u](y)$ can be done
to establish that if the limit $\lim_{y\to0}y^{1-2s}\partial_yU[u](y)$ exists then $u\in D((-L)^s)$.

Finally, $(c)$ follows from Lemma \ref{lem:extension_equation} and Theorem \ref{lem:unique_A_unbounded}.
\end{proof}

\begin{thm}[Extension problem for $1<s<2$]\label{thm:s in between 1 and 2}
Let $L$ be the infinitesimal generator of a uniformly bounded $C_0$-semigroup $\{e^{tL}\}_{t\geq0}$
on $X$. Assume that $0 \in \rho(L)$. For $1<s<2$ and $u\in X$, let $U(y)=U[u](y)$ be its extension as in \eqref{eq:Usemigroup}.
The following statements hold.
\begin{enumerate}[$(a)$]
\item If $u\in D((-L)^s)$ and $f=(-L)^su$ then
\begin{equation}\label{eq:U component s large}
\begin{aligned}
U(y) &= u+\frac{y^2 \Gamma(s-1)}{4 \Gamma(s)}Lu+\frac{1}{\Gamma(s)}\int^\infty_0
\bigg[e^{-y^2/(4t)}-1+\frac{y^2}{4t}\bigg]e^{tL} f\,\frac{dt}{t^{1-s}}  \\
&= u+\frac{y^2 \Gamma(s-1)}{4 \Gamma(s)}Lu+\frac{y^{2s}}{4^s\Gamma(s)}\int^\infty_0
\big(e^{-r}-1+r\big)e^{\frac{y^2}{4r}L}f\,\frac{dr}{r^{1+s}}.
\end{aligned}
\end{equation}
\item We have that $u\in D((-L)^s)$ if and only if the limits
\begin{equation}\label{eq:1s2incremental}
\lim_{y\to0}y^{3-2s}\partial_y \big(\tfrac{2}{y} \partial_y U(y) \big)\qquad\hbox{or}\qquad
\lim_{y \to 0}\frac{U(2y) - 4 U(y) + 3u}{y^{2s}}
\end{equation}
or
$$\lim_{y \to 0}y^{3-2s}\partial_y \big( LU +\tfrac{3-2s}{y} \partial_y U + \partial_{yy}U \big)$$
exist in $X$ and, in these cases,
\begin{equation}\label{eq:U Ls}
\lim_{y\to0^+}y^{3-2s}\partial_y \big(\tfrac{2}{y}\partial_yU(y)\big)=\lim_{y \to 0}y^{3-2s}\partial_y \big( LU +\tfrac{3-2s}{y} \partial_y U + \partial_{yy}U \big)=c_s(-L)^su
\end{equation}
and
\begin{equation}\label{eq:limincremental1s2}
\lim_{y \to 0}\frac{U(2y) - 4 U(y) + 3U(0)}{y^{2s}}=d_s(-L)^su
\end{equation}
where $c_s=\frac{\Gamma(2-s)}{4^{s-3/2}\Gamma(s)}$ and $d_s = \frac{(4^{1-s}-1)\Gamma(1-s)}{\Gamma(1+s)}$.
\item If $u\in D((-L)^s)$ then its extension $U$ given by \eqref{eq:Usemigroup}
is the unique classical solution to the fourth order initial value problem
\begin{equation}\label{eq:uniquenessL1s2}
\begin{cases}
\big(L+\frac{3-2s}{y}\partial_y+\partial_{yy}\big)^{2}U=0&\hbox{for}~y>0\\
\lim_{y\to0}U(y)=u&\hbox{in}~X \\
\lim_{y\to0}y^{3-2s}\partial_yU(y)=0&\hbox{in}~X \\
\lim_{y\to0}\big(LU+\frac{3-2s}{y}\partial_yU+\partial_{yy}U\big)=\frac{ \Gamma(s-1)}{ \Gamma(s)}Lu&\hbox{in}~X \\
\lim_{y\to0}y^{3-2s}\partial_y\big(LU+\frac{3-2s}{y}\partial_yU+\partial_{yy}U \big)=c_s(-L)^su &\hbox{in}~X
\end{cases}
\end{equation}
and to the second order initial value problem
\begin{equation}\label{eq:uniquenessy1s2}
\begin{cases}
LU+\frac{1-2s}{y}\partial_yU+\partial_{yy}U=0&\hbox{for}~y>0\\
\lim_{y\to0}U(y)=u&\hbox{in}~X \\
\lim_{y\to0}y^{3-2s}\partial_yU(y)=0&\hbox{in}~X \\
\lim_{y\to0}\frac{2}{y}\partial_yU(y)=\frac{ \Gamma(s-1)}{ \Gamma(s)}Lu&\hbox{in}~X \\
\lim_{y\to0}y^{3-2s}\partial_y\big(\frac{2}{y}\partial_yU(y)\big)=c_s(-L)^su &\hbox{in}~X.
\end{cases}
\end{equation}
\end{enumerate}
\end{thm}

\begin{proof}
To begin with $(a)$, let $u\in D((-L)^s)$ and $f=(-L)^su$.
By Theorem \ref{lem:approx_identity},
$$u=\lim_{\varepsilon\to0}(\varepsilon I-L)^{-s}f\qquad\hbox{and}\qquad -Lu=\lim_{\varepsilon\to0}(\varepsilon I-L)^{-(s-1)}f$$
Therefore, by Theorems \ref{lem: extension_epsilon_presentation} and \ref{lem:inverse} and the dominated convergence theorem,
\begin{align*}
U(y)-u-\frac{y^2 \Gamma(s-1)}{4 \Gamma(s)}Lu
    &= \lim_{\varepsilon\to 0}\frac{1}{\Gamma(s)}\int^\infty_0 e^{-\varepsilon t}
    \bigg[e^{-y^2/(4t)}-1+\frac{y^2}{4t}\bigg]e^{tL}f\,\frac{dt}{t^{1-s}} \\
    &= \frac{1}{\Gamma(s)}\int^\infty_0\bigg[e^{-y^2/(4t)}-1+\frac{y^2}{4t}\bigg]e^{tL}f\,\frac{dt}{t^{1-s}}.
\end{align*}
This and the change of variables $r=y^2/(4t)$ give \eqref{eq:U component s large}.

For $(b)$, suppose that $u\in D((-L)^s)$ and set $f=(-L)^su$. Notice that, by Lemmas \ref{lem:extension_equation}
and \ref{lem:recursion} with $a=1-2s$ and $n=m=1$,
\begin{equation}\label{eq:Landdy}
LU+\tfrac{3-2s}{y}\partial_yU+\partial_{yy}U=\tfrac{2}{y}\partial_yU.
\end{equation}
Differentiating the first formula in \eqref{eq:U component s large} and changing variables $r=y^2/(4t)$, we obtain
$$y^{3-2s}\partial_y\big(\tfrac{2}{y}\partial_yU\big)(y)= \frac{1}{4^{s-3/2}\Gamma(s)}\int^\infty_0   e^{-r}  e^{\frac{y^2}{4r}L}f\,\frac{dr}{r^{s-1}}$$
Using this and \eqref{eq:Landdy}, we get \eqref{eq:U Ls}. Next, by \eqref{eq:U component s large},
$$\frac{U(2y) - 4U(y) + 3U(0)}{y^{2s}} = \frac{1}{4^s\Gamma(s)}\int^\infty_0
\big(e^{-r}-1+r\big)\big(4^se^{\frac{y^2}{r}L}f-4e^{\frac{y^2}{4r}L}f\big)\,\frac{dr}{r^{1+s}}.$$
We can take the limit as $y\to0$ here and get \eqref{eq:limincremental1s2}.

Let us prove the converse of $(b)$. For any $v\in X$, let $U(y)=U[v](y)$
denote its extension as in \eqref{eq:Usemigroup}. Define the linear operator
$$T_s v =\lim_{y \to 0}y^{3-2s}\partial_y\big(\tfrac{2}{y}\partial_yU[v](y)\big)=
\lim_{y\to0}y^{3-2s}\partial_y\big(LU[v](y)+\tfrac{3-2s}{y} \partial_y U[v](y) + \partial_{yy}U[v](y) \big)$$
for $v\in D(T_s)=\{v\in X: T_sv~\hbox{exists}\}$, where in the second identity we applied \eqref{eq:Landdy}. We just proved that
$D((-L)^s)\subset D(T_s)$. For the opposite inclusion, let $u \in D(T_s)$ and $g = T_s u$.
Since $(-L)^{-s}$ is a bounded linear operator that commutes with $e^{tL}$ and $(-L)^su\in D((-L)^s)$,
$$(-L)^{-s} g=(-L)^{-s}T_su=T_s[(-L)^{-s}u]=c_s(-L)^s[(-L)^{-s}u]=c_su.$$
that is, $u \in D((-L)^s)$. 
A similar argument can be done to show that the second limit in \eqref{eq:1s2incremental}
exists then $u\in D((-L)^s)$, and the details are left to the interested reader. 

For $(c)$, Lemma \ref{lem:extension_equation} and Lemma \ref{lem:recursion} give that
$U$ solves the first equations in \eqref{eq:uniquenessy1s2} and \eqref{eq:uniquenessL1s2}.
By \eqref{eq:Landdy}, it is enough to verify that $U$ satisfies the boundary conditions of \eqref{eq:uniquenessy1s2}.
Differentiating the first formula in \eqref{eq:U component s large} and changing variables,
$$\partial_yU(y)=\frac{y \Gamma(s-1)}{2 \Gamma(s)}Lu - \frac{y^{2s-1}}{4^{s-1/2}\Gamma(s)}\int^\infty_0\big( e^{-r} - 1 \big)
e^{\frac{y^2}{4r} L}f\, \frac{dr}{r^{s}}$$
so that $\lim_{y\to0}y^{3-2s}\partial_yU(y)=0$ and $\lim_{y\to0}\frac{2}{y}\partial_yU(y)=\frac{\Gamma(s-1)}{\Gamma(s)}Lu$.
The uniqueness is established in Theorem \ref{cor:unique1s2}.
\end{proof}

\section{Proof of Theorem \ref{thm:extension-general}}\label{sec: extension s general}

For the proof of Theorem \ref{thm:extension-general} we fix $s>0$ with integer part $[s]$.
The cases $[s]=0$ and $[s]=1$ are contained in Theorems \ref{thm: extension frac domain} and
\ref{thm:s in between 1 and 2}, respectively. Thus, we consider $[s]\geq2$.

Lemmas \ref{lem:extension_equation} and \ref{lem:recursion} with $a=1-2s$ and $m=[s]$ give $(a)$.

For $(b)$, let $u\in D((-L)^s)$ and $f=(-L)^su$. Theorem \ref{lem:approx_identity} implies that, for any $0\leq k\leq[s]$,
$$\lim_{\vep\to0} (\vep I-L)^{-(s-k)}f=\lim_{\vep\to0} (\vep I-L)^{-(s-k)} (-L)^s u = (-L)^ku .$$
Using this, the semigroup expression for $(\vep I-L)^{-(s-k)}$ given in Theorem \ref{lem:inverse}
and the limit formula for $U$ in Theorem \ref{lem: extension_epsilon_presentation}, we can write
$$U(y) - \sum^{[s]}_{k=0}\bigg(\frac{y^2}{4}\bigg)^k \frac{\Gamma(s-k)}{k!\Gamma(s)}L^k u \\
=\lim_{\vep\to0}\frac{1}{\Gamma(s)}\int_0^\infty\bigg[e^{-y^2/(4t)}-\sum_{k=0}^{[s]}\frac{(-\frac{y^2}{4t})^k}{k!}\bigg] e^{-\vep t} e^{tL}f\,\frac{dt}{t^{1-s}}.$$
The limit can be placed inside the integral sign because its integrand is bounded by
$$C(s,y,M,\|f\|_X)\bigg[\frac{\chi_{(0,1)}}{t^{1-s}}+\frac{\chi_{(1,\infty)(t)}}{t^{2+[s]-s}}\bigg]\in L^1((0,\infty),dt)$$
uniformly in $\vep>0$. This and the change of variables $r=y^2/(4t)$ give \eqref{eq:U component s Greater 2}.

To prove $(c)$ and $(d)$, let $u\in D((-L)^s)$. Since $LU+\frac{1-2s}{y}\partial_yU+\partial_{yy}U=0$,
we can apply Lemma \ref{lem:recursion} with $a=1-2s$ and $0\leq m\leq n=[s]$ to get
$$\big(L+\tfrac{1-2(s-[s])}{y}\partial_y+\partial_{yy}\big)^mU=\frac{[s]!}{([s]-m)!}\big(\tfrac{2}{y} \partial_y\big)^mU.$$
In particular, it is enough to prove that \eqref{eq:U Ls Greater 2} exists and that $U$ solves
\eqref{eq:bvpinyalls}, as uniqueness follows from Theorem \ref{thm:uniquenessgeneral}.
For this, we define the functions
$$S_{n,N}(r) = \sum^N_{k=n}\frac{(-1)^{k-n}}{(k-n)!} r^{k-n} \frac{\Gamma(s-k)}{\Gamma(s)}(-L)^k u$$
and
$$F_N(r) = e^{-r} - \sum^N_{k=0}\frac{(-r)^k}{k!}$$
for $r>0$. Then
$$U(y)= S_{0,[s]}\big(\tfrac{y^2}{4}\big)+\frac{1}{\Gamma(s)} \int^\infty_0 F_{[s]}\big(\tfrac{y^2}{4t}\big)e^{tL}(-L)^s u \, \frac{dt}{t^{1-s}}.$$
Since $S_{n,N}'(r)=-S_{n+1,N}(r)$, by induction, for any $m\geq0$,
$$\big(\tfrac{2}{y} \partial_y \big)^mS_{n,N}\big(\tfrac{y^2}{4}\big)=(-1)^mS_{n+m,N}\big(\tfrac{y^2}{4}\big).$$
Similarly, $F'_N(r)=-F_{N-1}(r)$ gives by induction that
$$\big(\tfrac{2}{y}\partial_y\big)^mF_N\big(\tfrac{y^2}{4t}\big)=\tfrac{(-1)^m}{t^m}F_{N-m}\big(\tfrac{y^2}{4t}\big).$$
With these identities and the change of variables $r=y^2/(4t)$,
\begin{align*}
\big(\tfrac{2}{y}\partial_y\big)^mU(y)
&= (-1)^mS_{m,[s]}\big(\tfrac{y^2}{4}\big)+\frac{(-1)^m}{\Gamma(s)} \int^\infty_0 F_{[s]-m}\big(\tfrac{y^2}{4t}\big)e^{tL}(-L)^su\,\frac{dt}{t^{1-s+m}} \\
&= (-1)^mS_{m,[s]}\big(\tfrac{y^2}{4}\big)+\frac{(-1)^my^{2(s-m)}}{4^{s-m}\Gamma(s)} \int^\infty_0F_{[s]-m} (r)e^{\frac{y^2}{4r}L}(-L)^s u\,\frac{dr}{r^{1+s-m}}.
\end{align*}
From dominated convergence, it follows that, for any $0\leq m\leq [s]$,
$$\lim_{y\to0}\big(\tfrac{2}{y}\partial_y\big)^mU(y)=\frac{\Gamma(s-m)}{\Gamma(s)}L^mu$$
as the integral term vanishes in the limit. One can also easily see that, for $0\leq m<[s]$,
$$2\lim_{y\to0}y^{1-2(s-[s])}\partial_y\big(\tfrac{2}{y}\partial_y\big)^mU(y)=
\lim_{y\to0}y^{2-2(s-[s])}\big(\tfrac{2}{y}\partial_y\big)^{m+1}U(y)=0.$$
Finally, when $m=[s]$,
$$\big(\tfrac{2}{y}\partial_y\big)^{[s]}U(y)=\frac{\Gamma(s-[s])}{\Gamma(s)}(-L)^{[s]}u+
\frac{(-1)^{[s]}}{\Gamma(s)}\int_0^\infty\big(e^{-y^2/(4t)}-1\big)e^{tL}(-L)^su\,\frac{dt}{t^{1-(s-[s])}}.$$
We can differentiate inside the integral sign and change variables $r=y^2/(4t)$ to get
$$\partial_y\big(\tfrac{2}{y}\partial_y\big)^{[s]}U(y)=\frac{(-1)^{[s]+1}y^{-1+2(s-[s])}}{4^{s-([s]+1/2)}\Gamma(s)}
\int_0^\infty e^{-r}e^{\frac{y^2}{4t}L}(-L)^su\,\frac{dr}{r^{s-[s]}}.$$
By using dominated convergence, we conclude that
$$\lim_{y\to0}y^{1-2(s-[s])}\partial_y\big(\tfrac{2}{y}\partial_y\big)^{[s]}U(y)=
\frac{(-1)^{[s]+1}\Gamma([s]+1-s)}{4^{s-([s]+1/2)}\Gamma(s)}(-L)^su.$$
This finishes the proof of $(d)$ and the direct statement of $(c)$.

For the converse of $(c)$, given $v\in X$, let $U[v](y)$ denote its extension as in \eqref{eq:Usemigroup}.
Define the linear operator
\begin{align*}
T_sv &= \lim_{y\to0}y^{1-2(s-[s])}\partial_y\big(\big(\tfrac{2}{y}\partial_y\big)^{[s]} U[v](y) \big) \\
&=\frac{1}{[s]!}\lim_{y\to0}y^{1-2(s-[s])}\partial_y\big(\big(L+\tfrac{1-2(s-[s])}{y}\partial_y+\partial_{yy}\big)^{[s]}U[v](y)\big)
\end{align*}
with domain $D(T_s)=\big\{v\in X:T_sv~\hbox{exists}\big\}$. We already know that $D((-L)^s)\subset D(T_s)$.
Let $u\in D(T_s)$ and set $g=T_su$. Then, using the continuity of $(-L)^{-s}$,
\begin{align*}
(-L)^{-s}g &= \lim_{y\to0}y^{1-2(s-[s])}\partial_y \big( \big(\tfrac{2}{y}\partial_y\big)^{[s]} U[(-L)^su](y) \big) \\
&= c_s(-L)^s(-L)^{-s}u=c_su.
\end{align*}
Hence, $u\in D((-L)^s)$.\qed



\end{document}